\documentclass[11pt,a4paper]{article}
\usepackage[OT1]{fontenc}
\usepackage[latin1]{inputenc}
\usepackage[english]{babel}
\usepackage{amsfonts}
\usepackage{amsthm}
\usepackage{amsmath}
\usepackage{amscd}
\usepackage{latexsym}
\usepackage{amssymb}

\usepackage[bookmarks,pdfpagelabels,linkbordercolor={0.1 0.8 0.4}]{hyperref}

\newtheorem{teo}{Theorem}[section]
\newtheorem{prop}[teo]{Proposition}
\newtheorem{lem}[teo]{Lemma}
\newtheorem{coro}[teo]{Corollary}

\theoremstyle{definition}
\newtheorem{rem}[teo]{Remark}
\newtheorem{nota}[teo]{Notation}
\newtheorem{ejem}[teo]{Example}
\newtheorem{ejems}[teo]{Examples}

\newcommand{\PI}[2]{\left\langle #1 , #2 \right\rangle}

\def\b{{\cal B}}
\def\a{{\cal A}}

\def\gg{\mathbb{G}}

\def\b{{\cal B}}

\def\u2{U_2({\cal H})}

\def\o2{ {\cal O}_A}

\def\noi{\noindent}

\begin{document}

\title{The group of $L^2$-isometries on $H^1_0$\vspace*{0cm} }
\date{}
\author{E. Andruchow, E. Chiumiento and G. Larotonda \footnote{All authors are partially supported by Instituto Argentino de Matemática {\it Alberto P. Calder\'on} and CONICET}}

\maketitle

\abstract{Let $\Omega$ be an open subset of  $\mathbb R^n$. Let $L^2=L^2(\Omega,dx)$ and  $H^1_0=H^1_0(\Omega)$ be the standard  Lebesgue and Sobolev  spaces of complex-valued functions. The aim of this paper is to study the group $\gg$ of invertible operators on  $H^1_0$ which preserve the $L^2$-inner product. When $\Omega$ is bounded and $\partial \Omega$ is smooth, this group acts as the intertwiner of the $H^1_0$ solutions of the non-homogeneous Helmholtz equation $u-\Delta u=f$,  $u|_{\partial \Omega}=0$. We show that $\gg$   is a real Banach-Lie group, whose Lie algebra  is ($i$ times)   the space of symmetrizable operators. We discuss  the spectrum of operators belonging to $\gg$ by   means of examples. In particular,  we give an example  of an operator in $\gg$ whose spectrum is not contained in the unit circle.  We also study the one parameter subgroups of $\gg$. Curves of minimal length in $\gg$ are considered.  We introduce the subgroups  $\gg_p:=\gg \cap (I - \b_p(H^1_0))$, where $\b_p(H_0^1)$ is a Schatten ideal of operators on $H_0^1$. An invariant  (weak) Finsler metric is defined by the  $p$-norm of the Schatten ideal of operators of $L^2$. We prove  that any pair of operators $G_1 , G_2 \in \gg_p$ can be joined by a minimal curve of the form $\delta(t)=G_1 e^{itX}$, where $X$ is a symmetrizable operator in $\b_p(H^1_0)$.

\section{ Introduction}
Let $\Omega\subset \mathbb{R}^n$ be an open subset.  Denote by  $L^2=L^2(\Omega, dx)$  the   Lebesgue space of square-integrable functions endowed with its usual inner product $\PI{\, \cdot \,}{\, \cdot \,}$. Let  $H_0^1=H_0^1(\Omega)$ be the closure in the Sobolev norm  of the $C^{\infty}$ functions with compact support contained in $\Omega$.
In this paper, we study the group $\gg$ of invertible operators on $H^1_0$ that preserve the $L^2$-inner product:
$$
\gg=\{ \, G\in Gl(H^1_0) \, : \,  \PI{Gf}{Gg}=\PI{f}{g} \, \}.
$$
In the case where $\Omega=\mathbb{R}^n$, the group $\gg$ was introduced in \cite{CM11} in relation with the geometry of the  variational spaces arising in the many-particle Hartree-Fock theory.
One could give an abstract definition of $\gg$, involving a complex Hilbert space $H$ and a dense and continuously included subspace $E\subset H$ with their respective (non equivalent) inner products. However, we preferred this concrete setting given by the inclusion $H_0^1 \subset L^2$ because we shall deal mainly with examples.

From the definition of the group $\gg$, it is clear that  the theory of operators on spaces with two norms will play a central role in the study of this group. This theory was independently initiated by M. G. Krein \cite{k} and P. D. Lax \cite{l}. In Section \ref{2 norms} we recall the most useful results for our purposes.

 Our first results on the structure of $\gg$ are given in Section \ref{basics}. We prove that $\gg$  is a real Banach-Lie group equipped with norm  of the algebra of bounded operators $\b(H^1_0)$, and its Lie algebra $\Gamma$ can be identified with the real Banach space of operators $X \in \b(H^1_0)$ such that
$\PI{Xf}{g}=-\PI{f}{Xg}$ for any $f,g \in H^1_0$. Thus $i \, \Gamma$ is a well studied  class of   operators that naturally arises when one deals with spaces with two norms, which is usually known as  the class of \textit{symmetrizable operators}.
     An alternative description of $\gg$ is given by
\[ \gg=\{G\in Gl(H^1_0): G^*AG=A\}, \]
where  $A$ is the positive operator on $H^1_0$ satisfying $[Af,g]=\PI{f}{g}$ and $[ \, \cdot \, , \, \cdot \,  ]$ denotes the inner product on $H_0^1$.
In fact, when $\Omega$ is a bounded domain in $\mathbb{R}^n$ and $\partial \Omega$ is smooth, the operator $A$ is the solution operator of the Sturm-Liouville equation. On the other hand, note that any operator belonging to $\gg$ may be extended to an unitary operator on $L^2$. This extension procedure induces a norm continuous representation  of the group onto $L^2$, which  does not have a continuous inverse.

In  Section \ref{section spectra} we examine by examples different elementary aspects of $\gg$. For instance, we show that the  norm of an operator in $\gg$ can be arbitrarily large. In the general setting of operators on spaces with two norms, it is known that there exist symmetrizable operators with non real spectrum. Nevertheless, the  few examples of this fact do not apply to our concrete situation (see \cite{ barnes, dieudonne, k}). We present an example of a symmetrizable operator belonging to $i \Gamma$ with non real spectrum (Example \ref{non real}). In particular, this implies that the spectrum of operators in $\gg$ may not be contained in the unit circle. Another interesting problem is to determine if $\gg$ is an exponential group. It turns out that this property depends on the topology of the set $\Omega$  (see Proposition \ref{exp multiplication}; Example \ref{example exp}).

In Section \ref{one parameter} we investigate the  one parameter subgroups of $\gg$.  We construct a  norm continuous unitary representation  $\gg \to U(H^1_0)$, $G \mapsto U_G$ satisfying $U_GA^{1/2}=A^{1/2}G$. Then we study  the infinitesimal generators associated with this representation.

The results concerning  the metric geometry of $\gg$ are presented in Section \ref{Finsler metrics}.  A natural invariant Finsler metric in $\gg$ is provided by the usual operator norm of $\b(L^2)$. If one measures length of curves with this metric, $\gg$ behaves like a unitary group near the identity. Indeed, any operator $G$ in $\gg$ such that $\|G - I \| \leq 1$ can be joined by a minimal curve of the form $\delta(t)=e^{itX}$, where $X$ is a symmetrizable operator and $\| \, \cdot \, \|$ stands for the operator norm in  $\b(H^1_0)$.
Next we consider the following subgroups:
\[  \gg_p  = \gg \cap (I - \b_p(H^1_0)),    \]
where $\b_p(H^1_0)$ is a Schatten ideal of operators on $H^1_0$ ($1\leq p \leq \infty$). Essentially due to the fact that the logarithm of operators in $\gg_p$ is well defined, we are able to extend the afore-mentioned minimality result to a global result in $\gg_p$, where the Finsler metric is now  the $p$-norm of the Schatten ideal $\b_p(L^2)$ (see Theorem \ref{thm minimal curves}).

\begin{nota} We end this introduction by fixing some notation. The Sobolev space $H_0 ^1$ is a Hilbert space with the inner product given by
\[   [f,g]=  \int_{\Omega}f \bar{g} \, dx \, +  \,   \int_{\Omega} \nabla f \cdot \nabla \bar{g} \, dx.     \]
To avoid confusion among the several norms considered, we  denote by $|\ \, \cdot  \, |_1 \, (= [\, \cdot \, , \, \cdot \,]^{1/2})$ the norm on $H^1_0$,  $| \, \cdot  \, |_2 \, (=\PI{\, \cdot \,}{\, \cdot \,}^{1/2})$ the norm on $L^2$,
 $\|\ \, \cdot  \, \|$ the operator norm in $\b(H^1_0)$, and by $\|\ \, \cdot  \, \|_{\b(L^2)}$ the operator norm in $\b(L^2)$. If a given operator $X$ acts both on $L^2$ and $H^1_0$, we shall denote by $\sigma_{L^2}(X)$ its spectrum as an operator on $L^2$, and by $\sigma_{H^1_0}(X)$ its spectrum as an operator on $H^1_0$.
\end{nota}

\medskip

\section{Background on operators on spaces with two norms}\label{2 norms}

 Let $H$  be a Hilbert space with an inner product $\PI{\,\cdot \,}{\, \cdot \,}$ and the associated norm $\| \, \cdot \, \|$.    Let $(E, | \, \cdot \,|_E)$ be a Banach space.    Assume that $E$ is a dense linear subspace of $H$ and suppose that the norms satisfy $\| \, \cdot \, \| \leq C | \, \cdot \, |_E$ for some positive constant $C$. Throughout this article, we are interesting in the special case where $E=H_0^1$ and $H=L^2$.

 Let $\b(E)$ (resp. $\b(H)$) denote the algebra of bounded operators on $E$ (resp. $H$).
An operator $X$ in $\b(E)$ is said to be \textit{symmetrizable} if
\[  \PI{Xf}{g}=\PI{f}{Xg}, \, \, \, \, \, f, g \in E. \]
Given an operator $X \in \b(E)$, we denote by $\sigma_E(X)$  the spectrum of $X$ over $E$. We use the obvious notation $\sigma_H(X)$ for the spectrum of $X$ over $H$.   In the following theorem we collect the basics results on  symmetrizable operators.
\begin{teo}[M. G. Krein \cite{k}, P. D. Lax \cite{l}]\label{symmetrizable results}
Let $X$ be a symmetrizable operator. The following assertions hold:
\begin{enumerate}
\item[i)] $X$ is bounded as an operator on $H$.
\item[ii)] $\sigma_H(X) \subseteq \sigma_E(X)$.
\item[iii)] If $\lambda$ belongs to the point spectrum of $X$ as an operator on $E$, then $\lambda$ belongs to the point spectrum of $X$ as an operator on $H$. Moreover, the eigenspace $\ker(X- \lambda)$  over $E$ and $H$ is the same.
\item[iv)] If $X$ is a compact operator on $E$, then $X$  is a compact operator on $H$.
\end{enumerate}
\end{teo}
\begin{rem}
\noi It is not difficult to see that the two possible norms of a  symmetrizable operator $X$ satisfy
$  \|X\|_{{\mathcal B}(H)} \leq \| X\|_{{\mathcal B}(E)} \, . $
\end{rem}
\noi A more general approach to study operators on spaces with two norms can be found in \cite{gz}. Since any $f \in H$ determines a continuous functional $\PI{\, \cdot \,}{f}$ of the space $E^*$, it follows that $E \subseteq  H \subseteq E^*$. A bounded operator $X$ on $E$ is called \textit{proper}
if $X'(E) \subseteq E$, where $X'$ is the (Banach) adjoint of $X$. If $X$ is proper, $X^+$ denotes the restriction of $X'$ to $E$. It can be shown that $X^+$ is the restriction to $E$ of the adjoint on $H$.

\begin{teo}[I. C. Gohberg, M. K. Zambicki\v{\i} \cite{gz}]\label{proper op}
Let $X$ be a proper operator. The following assertions hold:
\begin{enumerate}
	\item[i)] $X$ is bounded as an operator on $H$.
	\item[ii)] $\sigma_H(X) \subseteq \sigma_E(X) \cup \overline{\sigma_E(X^+)}$, where the bar indicates complex conjugation.
  \item[iii)] If $X$ is a compact operator on $E$, then $X$  is a compact operator on $H$. Moreover, $\sigma_H(X)=\sigma_E(X)$ and the eigenspaces of $X$  in  $E$ and $H$ corresponding to the non zero eigenvalues  coincide.
\end{enumerate}
\end{teo}

\noi If $E$ is also a Hilbert space with an inner product denoted by $[\, \cdot \, , \,  \cdot \,]$, it follows that there is a positive operator $A$ on $E$ such that $[Af,g]=\PI{f}{g}$. Thus $X$ is symmetrizable if and only if $AX=X^*A$, where the adjoint is taken with respect to $E$. The following result will be useful.

\begin{teo}[J. Dieudonné \cite{dieudonne}]\label{Quasiherm}
Let $A$ be a positive operator on a Hilbert space $E$.   Let $X$ be a bounded operator on $E$ such $AX=X^*A$. Then there is a unique self-adjoint operator $Y$ on $E$ such that $A^{1/2}X=YA^{1/2}$.
\end{teo}

\section{Basic facts on $\gg$}\label{basics}

From the definition of the group, it follows that any operator in $\gg$ extends to an isometry of $L^2$, which has a dense subset in its range, namely $H^1_0$.  Then,  operators belonging to $\gg$ extend to  unitary operators onto $L^2$. Thus one can describe $\gg$ alternatively as
$$
\gg=\{ \, G=U|_{H^1_0} \, : \, U\in U(L^2) \hbox{ such that } U(H^1_0)=H^1_0 \, \}.
$$
Moreover, there is a third  algebraic characterization of $\gg$.
Note that the sesqui-linear form $\PI{\,\cdot \,}{\,\cdot \,}$ is  bounded and positive definite in $H^1_0$, thus there exists a positive operator $A\in\b(H^1_0)$ such that
$$
\PI{f}{g}=[Af,g]=[f,Ag].
$$
Therefore, a straightforward computation shows that
\begin{equation}\label{operator A}
\gg=\{G\in Gl(H^1_0): G^*AG=A\}.
\end{equation}
From this characterization it becomes apparent that $\gg$ is a closed subgroup of $Gl(H^1_0)$.
We shall see in Section \ref{smooth} that it is a Banach-Lie group, and a submanifold of $\b(H^1_0)$. Its Lie algebra is
$$
\Gamma=\{\, X\in\b(H^1_0) \, : \, X^*A+AX=0 \,\}.
$$
Note that $X\in\Gamma$ if
$$
\PI{Xf}{g}=[AXf,g]=-[X^*Af,g]=-[Af,Xg]=-\PI{f}{Xg},
$$
i.e. if $X$ is antihermitic for the $L^2$-inner product. Therefore one has the following spatial characterization of $\Gamma$:
$$
\Gamma=\{\, X=Z|_{H^1_0} \,:  \, Z\in \b(L^2), \, Z^*=-Z, \, Z(H^1_0)\subset H^1_0 \, \}.
$$
In fact, if $X=Z|_{H^1_0}$ as above, then the operator $X:H^1_0 \to H^1_0$ verifies $\PI{Xf}{g}=-\PI{f}{Xg}$ for $f,g\in H^1_0$, and therefore is bounded in $H^1_0$ by the uniform boundedness principle. Conversely, if $X\in\b(H^1_0)$ satisfies $\PI{Xf}{g}=-\PI{f}{Xg}$,
then $iX$ lies in $\b(H^1_0)$ and it is symmetric for the $L^2$-inner product. It follows by  Theorem \ref{symmetrizable results} that $iX$ extends to a bounded self-adjoint operator on $L^2$.

\medskip

In order to understand $\gg$, it will be useful to provide some examples of elements in $\gg$. As is standard notation, if $f,g\in H^1_0$, denote by $f\otimes g$ the rank one operator in $\b(H^1_0)$ given by $f\otimes g(h)=[h,g]f$. Apparently, $(f\otimes g)^*=g\otimes f$, $\|f\otimes g\|=|f|_1|g|_1$, and if $B,C\in\b(H^1_0)$, $B(f\otimes g)C=Bf \otimes C^*g$.
\begin{ejems}\label{ejemplos basicos}\hspace{-0.5cm}
\begin{enumerate}
\item A straightforward verification shows that a unitary operator $U$ on $H^1_0$ which commutes with $A$, belongs to $\gg$. Conversely, if a unitary operator on $H^1_0$ belongs to $\gg$, then it commutes with $A$.
\item
Let $f\in H^1_0$ such that $|f|_2=1$. Then $f\otimes Af$ is a rank one idempotent:
$$
(f\otimes Af)^2=(f\otimes Af(f))\otimes Af=([f,Af]f)\otimes Af=\PI{f}{f}f\otimes Af=f\otimes Af.
$$
Note that $f\otimes Af$ extends to an orthogonal projection on $L^2$, $if\otimes Af\in\Gamma$ and
$$
e^{if\otimes Af}=e^if\otimes Af+(1-f\otimes Af)\in \gg.
$$
By the above remarks, $f\otimes Af$ is  an orthogonal projection on $H^1_0$ if and only if $f$ is an eigenvector of $A$.
\item
Let ${\cal S}$ be a finite dimensional subspace of $H^1_0$, and let $f_1,...,f_k$ a basis of ${\cal S}$ which is orthonormal for the $L^2$-inner product $\PI{\,\cdot \,}{\, \cdot \,}$. Then there exists a closed subspace ${\cal T}$ of $H^1_0$ such that ${\cal S} + {\cal T}=H^1_0$, and ${\cal S},{\cal T}$ are orthogonal for $\PI{\,\cdot \,}{\, \cdot \,}$. Indeed, let
$$
E=\sum_{j=1}^kf_j\otimes Af_j \in \b(H^1_0).
$$
Note that $E(f)=\sum_{j=i}^k\PI{f}{f_j}f_j$, i.e. $E$ is the $L^2$ orthogonal projection onto ${\cal S}$. Then ${\cal T}=\ker(E)$. Let $U_0$ be an operator in $\b({\cal S})$, which is isometric for the $L^2$-norm $| \, \cdot \, |_2$, and put
$$
G:H^1_0 \to H^1_0 ,\   G|_{\cal S}=U_0 \ \hbox{ and } G|_{{\cal T}}=1_{{\cal T}}.
$$
Then it is easy to check that $G\in\gg$.

\medskip

The former two examples consist of operators which are of the form $1$ $+$ compact (in fact finite rank). Let us show two examples which  are not of this form: multiplication and composition operators.
\item
Let $H^{1,\infty}(\Omega)$ be the space of complex-valued functions in $L^{\infty}(\Omega)$ such that their first partial derivatives in the distributional sense also belong to $L^{\infty}(\Omega)$. Pick $\theta \in H^{1,\infty}(\Omega)$ satisfying $|\theta(x)|=1$, and consider $M_\theta$ defined by
$$
M_\theta f(x)=\theta(x)f(x) , \ x\in\Omega.
$$
Then, $M_\theta$ is a linear operator which acts both in $L^2$ and $H^1_0$. It is a unitary operator in $L^2$, and preserves $H^1_0$:
clearly $M_\theta(H^1_0)\subset H^1_0$, and $(M_\theta)^{-1}(H^1_0)=M_{\bar{\theta}}(H^1_0)\subset H^1_0$, i.e. $M_\theta(H^1_0)=H^1_0$. It follows that $M_\theta\in\gg$.

\medskip

\item  Let $\psi: \Omega \to \Omega$ be  a volume-preserving $C^1$ diffeomorphism such that the partial derivatives $\psi_{x_j}$ and  $(\psi^{-1})_{x_j}$, $j=1, \ldots, n$, are bounded on $\Omega$. It is not difficult to show  that the operator $U_{\psi}: H^1_0 \to H^1_0$, $U_{\psi}(f)=f \circ \psi$ is an isomorphism (see e.g. \cite[Proposition 2.47]{treves}). Moreover, it also satisfies that
\[ \int_{\Omega} | U_{\psi}(f) |^{2} \, dx = \int_{\Omega} | f  \circ \psi |^{2} \, dx = \int_{\Omega} | f  |^{2} \, | det (D\psi ^{-1})  | \, dx =     \int_{\Omega} |f|^{2} \, dx , \]
where we have used that  $| det (D\psi ^{-1}) | \equiv 1$.  This shows that  $U_{\psi} \in \gg$.

\end{enumerate}
\end{ejems}

\subsection{Smooth structure}\label{smooth}

Now we prove the preceding statement about the Lie algebra of $\gg$.

\begin{lem}
The Lie algebra of $\gg$ is $\Gamma$.
\end{lem}
\begin{proof}
In order to prove this assertion it suffices to show that
$$
\Gamma=\{Y\in\b(H^1_0): e^{tY}\in \gg \hbox{ for all } t\in\mathbb{R}\}.
$$
If $X\in\Gamma$, $(X^*)^kA=(-1)^kAX^k$. Then $(e^{tX})^*A=e^{tX^*}A=Ae^{-tX}=A(e^{tX})^{-1}$, i.e. $e^{tX}\in \gg$. Conversely, if $e^{tY}\in \gg$ for all $t$, we may differentiate the identity $e^{tY^*}A=Ae^{-tY}$ at $t=0$, to obtain $Y^*A=-AY$.
\end{proof}

\begin{lem}\label{lemaexponencial}
Let $G \in \gg$. The following assertions hold:
\begin{enumerate}
\item[i)] Let $\mathcal{L}$ be a half-line in the complex plane, from $0$ to infinity. If $\sigma_{H^1_0}(G)\cap \mathcal{L}=\emptyset$, then there exists $X \in \Gamma$ such that $e^X=G$.
\item[ii)] If $\|G-1\|\le 1$, then there exists $X\in\Gamma$ such that $e^X=G$.
\end{enumerate}
\end{lem}
\begin{proof}
$i)$ We first note that one can consider $e^{i\theta}G$ in place of $G$, where $\theta$ is a suitable angle,   to reduce the proof to the case where $\mathcal{L}$ is the negative real axis.  Thus we will assume that $\mathcal{L}$ is the negative real axis.

Since $0 \notin \sigma_{H^1_0}(G)$ and $0 \notin \sigma_{H^1_0}(G^{-1})$, then it is possible to find a simple closed curve $\gamma$, which does not intersect $\mathcal{L}$ and contains $\sigma_{H^1_0}(G)$ and $\sigma_{H^1_0}(G^{-1})$ in its interior. In addition, we can choose $\gamma$ satisfying $\overline{\gamma}=\gamma$.
From the assumption $\sigma_{H^1_0}(G)\cap \mathcal{L}=\emptyset$, it follows that there is a well defined branch of  the logarithm, and $X=\log(G)$ can be defined using the Riesz functional calculus. If $\gamma$ is counterclockwise oriented, then
\begin{align*}
X^*A & = - \frac{1}{2\pi i} \int_{\gamma} \overline{\log(z)} \, (G^* - \overline{z})^{-1}A \, dz
= - \frac{1}{2\pi i} \int_{\gamma} \log(\overline{z}) A (G^{-1} - \overline{z})^{-1} \, dz \\
& =\frac{1}{2\pi i} \int_{\overline{\gamma}} \log(z) A (G^{-1} - z)^{-1} \, dz = A \log (G^{-1})=-AX.
\end{align*}
Hence $X \in \Gamma$, and the proof is complete.

\smallskip

\noindent $ii)$ Under the assumption $\|G-1\| \leq 1$, we have that $\sigma_{H^1_0}(G)$ does not intersect the negative real axis (note that $0\notin \sigma_{H^1_0}(G)$). Then the result can be deduced from $i)$.
\end{proof}

\noi This lemma allows us to exhibit local charts for $\gg$, modeled on $\Gamma$:
\begin{prop}\label{B-Lie}
 The group $\gg$ is a  real Banach-Lie group endowed with the norm topology of $\b(H^1_0)$.
\end{prop}
\begin{proof}
Let us consider the open subsets $\mathcal{U}=\{   \, X \in \b(H^1_0)  \,  :  \,  \sigma_{H^1_0}(X) \subseteq \mathbb{R} + i (-\pi, \pi)  \, \}$ and $\mathcal{W}=\{   \, G \in Gl(H^1_0)   \,  :  \, \arg (z) \in (-\pi, \pi), \, \forall \, z \in \sigma_{H^1_0} (G)  \, \}$. The exponential map of $Gl(H^1_0)$, i.e.
$$\exp: \mathcal{U} \to \mathcal{W}, \, \, \, \exp(X)=\sum_{i=0}^{\infty} \frac{X^n}{n!}  $$
is a real analytic bijection (see \cite[Lemma 2.11]{Upmeier}).
According to Lemma \ref{lemaexponencial} $i)$, it follows that $\exp(\mathcal{U} \cap \gg)=\mathcal{W}\cap \Gamma$.   Then a standard translation procedure can be used to cover $\gg$. The smoothness of the group operations follows from that of the group operations in $Gl(H^1_0)$.
\end{proof}

\begin{rem}
In the case where $\Omega=\mathbb{R}^n$, it was shown in \cite{CM11} that $\gg$ is an algebraic subgroup of $Gl(H^1_0)$. Hence the
Banach-Lie structure of $\gg$ followed from a general result on algebraic subgroups  (see \cite[Theorem 1]{harriskaup77}). It is noteworthy that $\gg$ is an algebraic subgroup of $Gl(H^1_0)$ for any open set $\Omega$,  and thus by the same result on algebraic subgroups, we have another proof of the smooth structure of $\gg$.
\end{rem}

\subsection{The relationship with the equation $u-\Delta u=h$}\label{laplace equation}

In this section we assume that $\Omega$ is bounded and $\partial \Omega$ is smooth. Let $f, g: \Omega \to \mathbb{C}$ be  $C^\infty$ functions with compact support contained in $\Omega$.  Note that these functions can be smoothly extended to  $\mathbb{R}^n$ by setting to be zero on the  complement of $\Omega$. Then,
$$
\PI{f}{g}=[Af,g]=\PI{Af}{g}+\int_\Omega\nabla Af(x)\cdot \nabla \bar{g}(x) dx,
$$
and by Green's formula,
$$
\PI{f}{g}=\PI{f}{Ag}- \int_\Omega Af \, \Delta \bar{g} dx= \PI{f}{A(g-\Delta g)}.
$$
Since this holds for any smooth function $f$, it follows that $A(g-\Delta g)=g$. Thus, if we denote by $h=g-\Delta g$, then $Ah=g$.
In other words, if $g$ is the unique solution of the non-homogeneous Helmholtz equation
\begin{equation}\label{la ecuacion}
\left\{ \begin{array}{l} u-\Delta u=h, \\ u|_{\partial \Omega}=0, \end{array} \right.
\end{equation}
then $Ah=g$. If $h \in H_0 ^1$, then $Ah=u_h$ is the weak solution of (\ref{la ecuacion}). That is, $1-\Delta$ is the unbounded right inverse of $A$ in $H^2_0$, or equivalently, $A$ is the solution operator of equation (\ref{la ecuacion}).
These facts are certainly well known (see e.g. \cite{treves}). Moreover, if $\Omega$ is bounded and $\partial \Omega$ is smooth, then $A$ is compact.
If $G\in \gg$, the equality $G^*AG=A$ can be interpreted as follows:
$G^*u_{Gh}=u_h$, or putting $h=G^{-1}f$,
$$
G^*u_f=u_{G^{-1}f},
$$
which means that $G$ intertwines solutions of (\ref{la ecuacion}).



\begin{ejem}
One simple example in which $A$ can be explicitly computed occurs when $\Omega=(0,1)$.  Let us compute its eigenvalues: $u_h=Ah=\lambda h$ implies that
$u_h-u_h''=\frac{1}{\lambda}u_h$, i.e.
$$
\left\{ \begin{array}{l} u_h''+(\frac{1}{\lambda}-1)u_h=0 \\ u_h(0)=u_h(1)=0 \end{array} \right.
$$
Then, $\lambda =(k^2\pi^2+1)^{-1}$ with eigenfunction $\sin(k\pi x)$. If we normalize these eigenfunctions in $H^1_0$, we get $s_k(x)=\frac{\sqrt2}{\sqrt{k^2\pi^2+1}}\sin(k\pi x)$ and
$$
A=\sum_{k=1}^{\infty} \frac{1}{k^2\pi^2+1}\  s_k\otimes s_k.
$$
\end{ejem}


\begin{ejem}
Let $\Omega\subset\mathbb R^2$ denote the open disk $x^2 + y^2 <1$. In this example, the eigenvalues and eigenfunctions of the Laplacian can be expressed in terms of  the Bessel functions $J_m$ ($m \geq 0$), which are defined by
\[  J_m(s)=\bigg( \frac{s}{2} \bigg)^m \, \sum_{p=0}^{\infty} \frac{(-1)^p}{\Gamma(p+1)\Gamma(m+p+1)} \bigg( \frac{s}{2} \bigg)^{2p}, \]
where $\Gamma$ stands for the Euler Gamma function. We refer the reader to  \cite[Example 34.2]{treves} for a detailed solution of the homogeneous Helmholtz equation in this example. It can be shown that the eigenvalues of the Laplace operator are given by  $\lambda= z_{m,j}^2$ for $m=0, 1, \ldots$ and $j=1, 2 , \ldots$, where $z_{m,j}$ denotes the positive zeros of $J_m$. It is convenient to express  the corresponding eigenfunctions in polar coordinates:
\[
   e_{m,j}(r,\theta) = \left\{
   \begin{array}{l l}
     J_m(z_{m,j}r)e^{im \theta} & \quad \text{if $m>0$,}\\
     J_0(z_{0,j}r) & \quad \text{if $m=0$,}\\
      J_{-m}(z_{-m,j}r)e^{im \theta} & \quad \text{if $m<0$.}
   \end{array} \right.
 \]
\medskip
According to formulas (5.14.6) and (5.14.9) in \cite{lebedev},
\[
\int_0 ^1 r J_m(z_{m,j}r)J_m(z_{m,k}r) \, dr =
\left\{
   \begin{array}{l l}
     0 & \quad \text{if j $\neq$ k,}\\
       \frac{1}{2} J_{m+1}^2(z_{m,j})    & \quad \text{if $j=k$.}
   \end{array} \right.
\]
Then, it follows that
\[   |e_{m,j}|_2= \sqrt{\pi} \, | J_{m+1}(z_{m,j})|.  \]
Hence, the solution operator is given by
\[ A = \sum_{m=-\infty}^{\infty} \sum_{j=1}^{\infty} \frac{1}{1+ z_{m,j}^2} \, s_{m,j} \otimes s_{m,j}\,,  \]
where the eigenfunctions $s_{m,j}=\frac{\sqrt{1+ z_{m,j}^2}}{\sqrt{\pi} |J_{m+1}(z_{m,j})|} \,   e_{m,j}$ are normalized in $H_0^1$.
\end{ejem}

\begin{ejem}\label{ela}
In the case $\Omega=\mathbb{R}^n$, $A$ can be explicitly computed. It is well known that $H_0 ^1 (\mathbb{R}^n)=H ^1 (\mathbb{R}^n)$ (see e.g. \cite[Proposition 24.9]{treves}), where the latter is the space of functions in $L^2(\mathbb{R}^n)$ with first partial (distributional) derivatives also belonging to $L^2(\mathbb{R}^n)$. A function $f\in H^1(\mathbb{R}^n)$ if and only if $(1+|\xi|^2)^{1/2}\hat{f}(\xi)\in L^2(\mathbb{R}^n)$, where $\hat{f}$ denotes the Fourier transform of $f$, and the inner product is given by
$$
[f,g]=\int_{\mathbb{R}^n} (1+|\xi|^2)\hat{f}(\xi)\overline{\hat{g}(\xi)} d\xi.
$$
Therefore, the solution operator is given by
$$
\widehat{Af} (\xi)=\frac{1}{1+|\xi|^2}\hat{f}(\xi).
$$
\end{ejem}

\subsection{The extension map}

\noindent  As we stated in the introduction, we may identify  $\gg$ with the subgroup of the unitary group $U(L^2)$ given by
\[  \mathcal{U}_{H_0^1}(L^2):= \{ \,W \in U(L^2)   \, : \, W(H^1_0)=H^1_0      \,  \}.    \]
In fact, the map
\begin{equation}\label{mapa grupos}
     \gg \longrightarrow \mathcal{U}_{H_0^1}(L^2), \, \, \, \, G \mapsto \bar{G}\, ,
\end{equation}
is a bijection, where $\bar{G}$ denotes the unique unitary operator $\bar{G}$ acting on $L^2$ which extends the operator $G \in \gg$. In what follows, we will endow $\mathcal{U}_{H_0^1}(L^2)$ with the operator norm topology of $\mathcal{B}(L^2)$, while $\gg$ will be considered  with the operator norm topology of $\mathcal{B}(H^1_0)$.

\medskip

\begin{prop}\label{desi}
Let $G_1,G_2 \in \gg$, then
\[  \|   \bar{G_1} -  \bar{G_2} \|_{\mathcal{B}(L^2)} \leq    \max\{ \,  \| G_1 ^{-1} \|  \, , \, \| G_2 ^{-1}  \| \,  \} \, \| G_1 - G_2 \|.   \]
In particular, the map in (\ref{mapa grupos}) is continuous. The inverse of this map is not continuous.
	\end{prop}
\begin{proof} Note that operators in $\gg$ are proper (see Section \ref{2 norms}).  In particular, note that $G^+=G^{-1}$ for any $G \in \gg$. Denote by $r_{H^1_0}(X)$   and $r_{L^2}(Y)$, respectively,      the spectral radius of $X \in \b(H^1_0)$ and   $Y \in \b(L^2)$. Then,
\begin{align*}
\|  \bar{G_1} -  \bar{G_2} \|_{\b(L^2)} & = \| 1 -  \bar{G_1} ^{-1}  \bar{G_2} \|_{\b(L^2)} \\
& = r _{L^2}(\, 1 -  \bar{G_1} ^{-1}  \bar{G_2}  \, ) \hspace{1cm} \text{(since $1 -  \bar{G_1} ^{-1}  \bar{G_2}$ is normal)}\\
& \leq \max\{  \, r_{H^1} (\, 1 - G_1 ^{-1} G_2    \,) \,  , \,  r_{H^1} (\, (1 - G_1 ^{-1} G_2)^+)    \} \hspace{0.5cm} \text{(by Theorem \ref{proper op} $ii)$)}\\
& = \max\{  \, r_{H^1} (\, 1 - G_1 ^{-1} G_2    \,) \,  , \,  r_{H^1} (\, 1 - G_2^{-1} G_1\,)    \} \\
& \leq \max\{  \, \| 1 - G_1 ^{-1} G_2   \| \,  , \,   \|1 - G_2 ^{-1} G_1 \|   \, \} \\
& \leq \max\{ \,\| G_1^{-1} \| \, , \, \|  G_2^{-1} \|  \,  \} \, \| G_1 - G_2 \|\,.
\end{align*}
Combining the preceding inequality with the continuity of the inversion map on $\gg$ gives that  the map in (\ref{mapa grupos}) is continuous.

Now we are going to prove that the inverse of the extension map is not continuous for any open subset $\Omega$ of $\mathbb{R}^n$. Let $x \in \Omega$ and $C:=(a_1,b_1) \times \ldots \times (a_n,b_n) \subset \Omega$ be a neighborhood of $x$.  Consider the following sequence of smooth functions
$$\theta_n: \Omega \to \mathbb{C}, \, \, \, \, \,  \theta_n(x_1, \ldots,  x_n)=e^{i \, \frac{\sin(nx_1)}{n}}.$$
Then, as we remarked in the fourth example of \ref{ejemplos basicos}, the multiplication operators $M_{\theta_n}$ belong to $\gg$. Given any $f \in L^2$ such that $|f|_2=1$, note that
\begin{align*}
| M_{\theta_n} (f) - f |_2 ^2 & = \int_{\Omega} |e^{i \, \frac{\sin(nx_1)}{n}} - 1|^2 \, |f(x)|^2 \, dx \\
& = 2 \int_{\Omega} \bigg(1 - \cos\bigg( \frac{\sin(nx_1)}{n} \bigg)  \bigg) \, |f(x)|^2 \, dx \leq
2 \, \bigg\|1 - \cos\bigg( \frac{\sin(nx_1)}{n} \bigg)   \bigg\|_{\infty} \to 0.
\end{align*}
Thus, $\| M_{\theta_n} - I \|_{\b(L^2)} \to 0$. On the other hand, let $f$ be a $C^{\infty}$ function with compact support such that $f(x) \equiv 1$ for $x \in C$.
We have
\begin{align*}
| M_{\theta_n} (f) - f |_1 ^2 & \geq  \int_C \nabla \theta_n f \cdot \nabla \bar{\theta}_n \bar{f} \, dx = \int_C \cos^2(n x_1) \, dx \\
& = (b_2 - a_2)\ldots (b_n - a_n) \bigg( \frac{1}{2n} \cos(nx_1)\sin(nx_1) + \frac{x_1}{2} \bigg|_{a_1} ^{b_1}    \bigg) \\
& \to \frac{1}{2}(b_1 - a_1)(b_2 - a_2)\ldots (b_n - a_n) >0,
\end{align*}
so that $\| M_{\theta_n} - I \| \nrightarrow 0$, and this shows that the inverse of the extension map is not continuous.
\end{proof}

\section{Norms and spectra of elements in $\gg$}\label{section spectra}

Note that if $G\in\gg$, the equality $G^*AG=A$ implies that $\|A\|\le \|A\|\|G\|^2$, and thus $\|G\|\ge 1$. Examining the previous examples, it can be shown that there are elements in $G$ with arbitrarily large norm.
\begin{ejems}
\begin{enumerate}
\item
Consider $\Omega=(0,1)$ and pick $f(x)= \sin(\pi x)+\sin(k\pi x)$. Clearly $|f|_2=1$. Thus, as in the second example of the first section, $G=e^{if\otimes Af}=e^if\otimes Af+(1-f\otimes Af)\in \gg$. Apparently,
$$
\|G\|\ge \max\{\|f\otimes Af\|, \|1-f\otimes Af\|\} \ge \|f\otimes Af\|=|f|_1|Af|_1 \,.
$$
A straightforward computation shows that
$$
|f|^2_1 |Af|^2_1=\frac14 \bigg(2+\pi^2(k^2+1)\bigg) \bigg( \frac{1}{\pi^2+1}+\frac{1}{k^2\pi^2+1} \bigg).
$$
Therefore, for large $k$, the norm of $G$ can be arbitrarily big.
\item Let $\Omega$ be an open and bounded subset of $\mathbb{R}^n$ such that $\partial \Omega$ is smooth. Under this assumptions, the operator $A$ is compact. Let $f \in H_0^1$ such that $|f|_1=1$ and $Af=\lambda f$, for some $\lambda \neq 0$.
Set $\theta:\Omega\subset \mathbb{R}^n\to \mathbb{C}$, $\theta(x)=e^{ik(x_1+...+x_n)}$ for some $k \in \mathbb{R}$. Then,
\begin{align*}
|M_\theta f|^2 _1 & = \PI{\theta f}{\theta f}+ \int_\Omega \nabla \theta f \cdot \nabla \bar{\theta} \,  \bar{f}  \, dx \\
& =  \PI{f}{f}+\int_\Omega  |f|^2 \, |\nabla  \theta |^2 + 2 Re \bar{f} \, \theta \, \nabla f \cdot \nabla \bar{\theta} + |\nabla f|^2 \, dx\\
&=[f,f]+\int_\Omega nk^2 |f|^2 +2f \, \nabla f \cdot \vec{k} \, dx
=1+nk^2 \lambda +k\sum_{j=1}^n\int_\Omega f  \frac{\partial f}{\partial x_j} \, dx,
\end{align*}
where $\vec{k}=(k,...,k)$. Since $f \in H_0^1$, integrating by parts,
$$
\int_\Omega f \frac{\partial f}{\partial x_j} \, dx=-\int_\Omega \frac{\partial f}{\partial x_j} f\, dx=0.
$$
Therefore,
$$
\|M_\theta\|\ge \sqrt{1+nk^2 \lambda}.
$$
\end{enumerate}
\end{ejems}

\medskip

Since any operator $G \in \gg$ can be extended to a unitary operator $\bar{G}$ on $L^2$ such that $\bar{G}(H^1_0)=H^1_0$, it is clear that  operators in $\gg$ are proper and $G^+=G^{-1}$. Thus by Theorem \ref{proper op} we know that
\begin{equation}\label{espectros}
\sigma_{L^2}(\bar{G})\subset \sigma_{H^1_0}(G)\cup \overline{\sigma_{H^1_0}(G^{-1})}.
\end{equation}
Let us examine now the spectra of the examples  \ref{ejemplos basicos}.
\begin{ejems}
\begin{enumerate}
\item
If $U$ is a unitary in $H^1_0$ which commutes with $A$, then clearly
$$
\sigma_{H^1_0}(U)=\sigma_{L^2}(U),
$$
and  it is a subset of $\mathbb{T}=\{z\in\mathbb{C}: |z|=1\}$.
\item
Examples 2 and 3 in Section \ref{ejemplos basicos} are constructed as operators $G$ acting on a $L^2$-orthogonal decomposition of $H^1_0$, $H^1_0={\cal S} + {\cal T}$, with ${\cal S}$ finite dimensional, and the operators acting as the identity on ${\cal T}$. Therefore their spectra in $\b(H^1_0)$ are finite, and consist of eigenvalues, and therefore are also eigenvalues of the extension $\bar{G}$ of $G$ to $L^2$. In particular they are elements of $\mathbb{T}$. It is apparent by construction, that also in this case both spectra coincide.
\item
As in Example \ref{ejemplos basicos}.4, consider $M_\theta$, where $\theta:\Omega\to \mathbb{C}$ is an element of $H^{1,\infty}(\Omega)$, now with $|\theta(x)|=1$. Clearly, $\sigma_{L^2}(M_\theta)=\mathcal{R}(\theta)\subset \mathbb{T}$, the essential range of $\theta$. Also in this case the spectra coincide (though none of the elements are eigenvalues). Indeed, if $M_\theta-\lambda =M_{\theta-\lambda}$ is invertible in $\b(L^2)$, then the function $\theta -\lambda$ does not vanish in $\Omega$, and moreover $(\theta-\lambda)^{-1}$ is also in $H^{1,\infty}(\Omega)$. Then the operator $M_{(\theta-\lambda)^{-1}}$, the inverse of $M_\theta-\lambda $ on $L^2$, defines a bounded operator in $H^1_0$, and therefore $M_\theta-\lambda $ is invertible on $H_0^1$.
Conversely, suppose that $M_{\theta}-\lambda $ is invertible in $\b(H^1_0)$, and let $B$ be its inverse. Since $M_{\theta} - \lambda$ is proper, it follows that
$B$ is proper if and only if $M_{\bar{\theta}} - \bar{\lambda}= (M_{\theta}-\lambda )^+$ is invertible in $\b(H_0^1)$ (see \cite[p. 148]{gz}). Using that $M_{\theta}-\lambda $ is bijective, it is straightforward to show that $M_{\bar{\theta}} - \bar{\lambda}$ is also bijective. By the open mapping theorem, $M_{\bar{\theta}} - \bar{\lambda}$
is invertible in $B(H_0^1)$. Therefore $B$ is proper, and by Theorem \ref{proper op}, it has a bounded extension to $L^2$. Hence $\lambda \notin \sigma_{L^2}(M_{\theta})$.
\end{enumerate}
\end{ejems}

\begin{ejem}
The examples of elements of $\gg$ so far have spectra   in $\mathbb{T}$. There is an example by Gohberg and Zambicki\v{\i} \cite{gz}, adapted by Barnes in \cite{barnes} to the case of a pair of Hilbert space norms, of an operator whose extension is symmetric, but whose spectrum does not lie in the real line. Namely, in this latter form, Barnes considers the Hilbert space $\ell^2$, and the dense subspace $\ell_0^2$, consisting of sequences $(a_n)_n$ such that $\sum_{n= 1} ^{\infty} 4^na_n^2<\infty$. Comparing with our situation, one has $[\ , \ ]$ in $\ell_0^2$, given by
$$
[a,b]=\sum_{n=1}^{\infty} 4^na_n\bar{b}_n,
$$
which makes $\ell_0^2$ a (complete)  Hilbert space, and the usual inner product $\PI{\, \cdot \,}{\, \cdot \,}$  of $\ell^2$, which is bounded in $\ell_0^2$. This latter inner product is implemented by a diagonal compact operator $A$, whose eigenvalues are $\frac{1}{4^n}$.


The above counterexample does not apply to our situation, where the operator $A$ is the solution operator. Let us reconstruct below the analogue of Barnes' example, and show that in our context, its spectrum is real. Consider $\Omega=(0,1)$,
$$e_k(x)=\sqrt2 \sin(k\pi x) \ , \hbox{ and }\  s_k(x)=\frac{1}{\gamma_k}e_k(x),
$$
the eigenvectors of $A$, normalized, respectively, in $L^2$ and $H^1_0$ (where $\gamma_k=\sqrt{k^2\pi^2 + 1}$). Let $T=S+B$ in $L^2$, where $S$ is the unilateral shift and $B$ is the backward-shift. Thus $T$ is self-adjoint in $L^2$, and $\sigma_{L^2}(T)\subset \mathbb{R}$. Apparently, $T(H^1_0)\subset H^1_0$. Indeed,
$$
S(s_k)=\frac{1}{\gamma_k}S(e_k)=\frac{1}{\gamma_k}(e_{k+1})=\frac{\gamma_{k+1}}{\gamma_k}s_{k+1}.
$$
Analogously $B(s_k)=\frac{\gamma_{k-1}}{\gamma_k}s_{k-1}$ (putting $e_0=s_0=0$).
Thus if $f=\sum_{k=1}^Nc_ks_k$,
$$
|Sf|_1 ^2=\big| \,\sum_{k=1}^N c_k\frac{\gamma_{k+1}}{\gamma_k}s_{k+1}\, \big|_1 ^2=\sum_{k=1}^N |c_k|^2\frac{\gamma_{k+1}^2}{\gamma_k^2}\, .
$$
The fractions $\frac{\gamma_{k+1}^2}{\gamma_k^2}$ are bounded by $2$. Thus
$$
|Sf|^2 _1 \le 2 \sum_{k=1}^N |c_k|^2= 2|f|^2 _1.
$$
It follows that $S$ is bounded in $H^1_0$ (and $\|S|_{H^1_0}\|\le \sqrt2$). Analogously $B$ is bounded in $H^1_0$ (with $\|B|_{H^1_0}\|\le 1$, because $\frac{\gamma_{k-1}^2}{\gamma_k^2}\le 1$). We claim that the spectrum of $T$ in $H^1_0$ is real, and coincides with its spectrum in $L^2$ (the analogous of $T$ in $\ell^2_0$ has non real spectrum). Indeed,  let $T'$ in $H^1_0$ be given by
$$
T'(s_k)=s_{k-1}+s_{k+1}.
$$
Clearly $T'$ is self-adjoint in $H^1_0$. Let $T'_N$ be given by
$T'_N(s_k)$ equal to $T(s_k)$, if $k\le N$, and to $T'(s_k)$ if $k\ge N+1$. Since $T'$ and $T'_N$ differ on a finite dimensional subspace, their essential spectra coincide: $\sigma_e(T'_N)=\sigma_e(T')\subset \mathbb{R}$. On the other hand $T-T'_N(s_k)=0$ if $k\le N$, and
$$
T-T'_N(s_k)=(\frac{\gamma_{k-1}}{\gamma_k}-1 )s_{k-1}+( \frac{\gamma_{k+1}}{\gamma_k}-1 )s_{k+1},
$$
if $k\ge N+1$. In our case, where $\Omega=(0,1)$, these fractions tend to $1$. Therefore $\|T-T'_N\|$ tends to $0$. By the semicontinuity property of the (essential spectrum), this implies that $\sigma_e(T)\subset \mathbb{R}$. It was proved in \cite{barnes}, that an extendable (or proper) operator, whose extension is self-adjoint, such as $T$, has the property that
$\sigma_{H^1_0}(T)\setminus\sigma_e(T)$ consists of isolated eigenvalues of finite multiplicity. As remarked before, these eigenvalues are necessarily real. It follows that $\sigma_{H^1_0}(T)\subset \mathbb{R}$.
\end{ejem}

However, modifying the example above one can obtain an element of $\gg$ whose spectrum as an operator of $H^1_0(\Omega)$ is not contained in $\mathbb{T}$.

\begin{ejem}\label{non real}
Let $\Omega$ be a bounded domain in $\mathbb{R}^n$ such that $\partial \Omega$ is smooth. We will show an example of a symmetrizable operator belonging to $i \Gamma$ with non real spectrum. In particular, this implies the existence of an operator in $\mathbb{G}$ with spectrum not contained in $\mathbb{T}$.

Let $A$ be the solution operator of equation (\ref{la ecuacion}), whose eigenvalues are related to the eigenvalues of the Laplacian in $\Omega$. It is a classical result by Hermann Weyl  in 1911 \cite{weyl}, that the eigenvalues of the Laplacian of a bounded domain $\Omega$ in $\mathbb{R}^n$ grow as
$$
\mu_k \sim  4 \pi \bigg(\frac{\Gamma(\frac{n}{2} + 1)}{|\Omega|}\bigg)^{2/n} \, k^{2/n},
$$
as $k\to \infty$. Since $\Omega$ is bounded and $\partial \Omega$ smooth, $A$ is compact, and consequently, there exists an orthonormal basis of eigenfunctions $(e_k)_k$   of $L^2$. Moreover, the eigenfunctions $(e_k)_k$ belong to $H^1_0$. By a straightforward computation taking into account the relationship between the $L^2$ and $H^1_0$ inner products,  it follows that  $s_k=\frac{e_k}{\gamma_k}$ is an orthonormal basis of $H^1_0$, where $\gamma_k=\sqrt{1 + \mu_k}$.


The orthonormal basis $(e_k)_k$ can be used to define the following bounded operator
$$S: L^2 \to L^2,  \, \, \, \, \, \, \, \, \, \, \, \, S(e_k)=e_{2k}\, .$$
Set $B=S^*$. Note that
$$ B(e_k)=\left\{ \begin{array}{l} 0, \, \,\, \,\, \,\, \, \,   \, \, \, \, \text{for $k$ odd,} \\ e_{k/2}, \, \, \, \, \text{ for $k$ even.} \end{array} \right.  $$
Then $T= S + B$ is a self-adjoint operator on $L^2$, so that $\sigma_{L^2}(T) \subseteq \mathbb{R}$. On the other hand,  for any $f \in H^1_0$, $f=\sum_{k=1}^{\infty}c_k s_k$, it is easily seen that
\[  |Sf|^2 _1= \bigg| \sum_{k=1}^{\infty} c_k \frac{\gamma_{2k}}{\gamma_k} s_{2k} \bigg|^2 _1 = \sum_{k=1}^{\infty} | c_k \frac{\gamma_{2k}}{\gamma_k} |^2 \leq K |f|^2 _1,  \]
where $K$ is a constant that bounds the convergent sequence $(\gamma_{2k} / \gamma_k)_k$. In a similar fashion, one can see that $B$ is bounded on $H^1_0$.
Hence $T(H^1_0) \subseteq H^1_0$, and $T$ turns out to be bounded on $H^1_0$. The expression of $T$ in the orthonormal basis of $H^1_0$ is given by
\[ T(s_k)=
\left\{ \begin{array}{l} \frac{\gamma_{2k}}{\gamma_k} \, s_{2k}, \, \,\, \, \,\, \, \,\, \,\, \,\, \, \, \,\, \,\, \,\, \, \, \,\, \,\, \,\, \, \,   \, \, \, \, \text{for $k$ odd,} \\
\frac{\gamma_{2k}}{\gamma_k}s_{2k} +  \frac{\gamma_{k/2}}{\gamma_k} \, s_{k/2}, \, \, \, \, \text{ for $k$ even.}
\end{array} \right. \]
We claim that $\sigma_{H^1_0}(T)$ contains all the points inside and on the ellipse
\[  \lambda= \sqrt[n]{2}e^{i\theta} + \frac{1}{\sqrt[n]{2}}e^{-i\theta}, \, \, \, \, \, \theta \in [0,2\pi]. \]
To this end, note that due to Weyl's asymptotic formula,
$$ \lim_{k \to \infty} \frac{\gamma_{2k}}{\gamma_k}=\sqrt[n]{2} \, \, \, \, \, \, \text{and} \, \, \, \, \, \, \lim_{k \to \infty} \frac{\gamma_{k/2}}{\gamma_k}=\frac{1}{\sqrt[n]{2}}.$$
Consider

\[ T'(s_k)=
\left\{ \begin{array}{l}  \sqrt[n]{2} s_{2k}, \, \,\, \, \,\, \, \,\, \,\, \,\, \, \, \,\, \,\, \,\, \, \, \,\, \,\, \,\, \, \,   \, \, \, \, \text{for $k$ odd,} \\
\sqrt[n]{2} s_{2k} +  \frac{1}{\sqrt[n]{2}} \, s_{k/2}, \, \, \, \, \text{ for $k$ even.}
\end{array} \right. \]
Then  the operators $T'_N$ defined by $T_N '(s_k)=T(s_k)$ if $k \leq N$, and $T_N ' (s_k)=T' (s_k)$ if $k \geq N+1$, satisfy $\| T_N ' - T\| \to 0$. Thus $\sigma_e(T)= \sigma_e(T_N ')=\sigma_e(T ')$, by the semicontinuity of the essential spectrum. So what is left is to show that the ellipse is contained in $\sigma_e(T ')$.  To prove the latter, note that the subspace
 \[  \mathcal{S}=\text{span}\{ \,  s_{2^k}   \, : \,  k \geq 0 \,    \}  \]
reduces $T'$. Then it follows that $\sigma_e(P_\mathcal{S}T'|_\mathcal{S}) \subseteq \sigma_e(T')$, where $P_\mathcal{S}$ denotes the orthogonal projection onto $\mathcal{S}$. But $P_\mathcal{S}T'|_\mathcal{S}$ is a Toeplitz operator  with $\sqrt[n]{2}$ under the diagonal and $1/ \sqrt[n]{2}$ over the diagonal. Thus $\sigma_e(T ')=\sigma_{H^1_0}(P_\mathcal{S}T'|_\mathcal{S})$, and according to a result by M. G. Krein \cite[Theorem 13.2]{MGKrein}, the spectrum of this Toeplitz operator is the above defined ellipse.
\end{ejem}

\subsection{Image of the exponential map}\label{Image}

The first examples of elements in $\gg$ given in Section \ref{ejemplos basicos} were operators $G=G_0\oplus I_{\cal T}$, where $G_0$ acts in ${\cal S}$ with $dim {\cal S}<\infty$, and ${\cal S}+{\cal T}=H^1_0$ an $L^2$-orthogonal sum.

\begin{prop}
If $G=G_0\oplus I_{\cal T}\in \gg$ as above, then there exists a finite rank operator $Z\in\Gamma$ such that
$$
G=e^Z.
$$
\end{prop}
\begin{proof}
Note that $G=1+F$, where $F$ has finite rank (inside ${\cal S}$). This implies that the spectrum of $G$ is finite and  consists of eigenvalues of modulus one and finite multiplicity. According  to Theorem \ref{proper op} $iii)$, we know that $\sigma_{L^2}(G)=\sigma_{H_0^1}(G)$ and the multiplicity of each
non zero eigenvalue coincide .  Therefore there exists  a self-adjoint operator $Z$  of finite rank in $L^2$ such that $\bar{G}=e^{iZ}$. Note that the eigenvectors of $\bar{G}$ are eigenvectors of $G$, so that the eigenvectors of $Z$ lie in $H^1_0$ and they are finite. Then $Z(H^1_0)\subset H^1_0$, and thus
$
X=iZ|_{H^1_0}\in \Gamma \hbox{ with } G=e^X.
$
\end{proof}

We point out a simple necessary condition on the spectrum of an operator in $\gg$ that belongs  to the image of the exponential map.

\begin{rem}
If $G=e^X$, with $X\in \Gamma$, we claim that $\sigma_{L^2}(\bar{G})\subset \sigma_{H^1_0}(G)$. Indeed,
let $Z \in \mathcal{B}(L^2)$ such that $Z^*=-Z$ and $Z|_{H^1_0}=X$. Recall that from  \cite[Theorem 2]{k} we have that $\sigma_{L^2}(Z)\subseteq\sigma_{H^1_0}(X)$. By the former set inclusion and a repeated application of the analytic spectral mapping theorem we find that
\[ \sigma_{L^2}(U_G)= \{  \,   e^{\lambda} \, : \, \lambda \in \sigma_{L^2}(Z) \} \subseteq \{  \,   e^{\lambda} \, : \, \lambda \in \sigma_{H^1_0}(X) \} = \sigma_{H^1_0}(G), \]
which proves our claim.
\end{rem}

\noi Next we study multiplication operators when the set $\Omega$ is bounded. The following lemma about functions is probably well known, but we give a proof bellow.

\begin{lem}\label{lips}
Let $\Omega$ be a bounded, connected and  open subset of $\mathbb{R}^n$. If $\theta \in H^{1,\infty}(\Omega)$, then $\theta$ is Lipschitz on $\Omega$.
\end{lem}
\begin{proof}
We may suppose that $n=2$. Take two arbitrary points  $(x,y), (\bar{x},\bar{y}) \in \Omega$. There exists a continuous curve $\gamma:[0,1] \to \Omega$ such that
$\gamma(0)=(x,y)$ and $\gamma(1)=(\bar{x},\bar{y})$. Then it possible to approximate $\gamma$ by a polygonal with segments parallel to coordinate axis. Moreover, each segment can be chosen inside of $\Omega$.  Let $(x,y)=(x_1,y_1), (x_2,y_1), (x_2,y_2), \ldots ,(x_m,y_m)=(\bar{x},\bar{y})$ denote the vertices of the polygonal. Now recall that the function $\theta$ is locally Lipschitz because  $\theta \in H^{1, \infty}(\Omega)$ (see \cite[p. 131]{EvG}). Therefore $\theta$ has partial derivatives almost everywhere, and it also holds that locally the function can be written as the integral of these partial derivatives. Then note that
\[  | \theta(x_{j+1} , y_j) - \theta(x_j , y_j) | \leq \int_{x_j} ^{x_{j+1}} \big| \frac{\partial \theta}{\partial x}(x,y_j) \, dx   \big| \leq \big\| \frac{\partial \theta}{\partial x} \big\|_{\Omega, \infty} \, (x_{j+1} - x_j). \]
A similar estimate holds for the other partial derivative. Since there are always a finite number of steps from $(x,y)$ to $(\bar{x},\bar{y})$, we get that
\begin{align*}
| \theta(x,y) - \theta(\bar{x},\bar{y}) | & \leq \sum_{j=1}^{m-1}   | \theta(x_j,y_j) - \theta(x_{j+1}, y_j) | + \sum_{j=2}^{m}   | \theta(x_j,y_{j-1}) - \theta(x_{j}, y_j) |  \\
& \leq  \big\| \frac{\partial \theta}{\partial x} \big\|_{\Omega, \infty}  |\bar{x} -x | + \big\| \frac{\partial \theta}{\partial y} \big\|_{\Omega, \infty} |\bar{y}- y|  \\
& \leq \sqrt{2} \, diam(\Omega) \, \max \{ \, \big\| \frac{\partial \theta}{\partial x} \big\|_{\Omega, \infty} \, , \,\big\| \frac{\partial \theta}{\partial y} \big\|_{\Omega, \infty} \, \} \, \| (x,y) - (\bar{x},\bar{y})\|.
\end{align*}\end{proof}

\begin{prop}\label{exp multiplication}
If $\Omega$ is bounded and connected, its closure $\overline{\Omega}$ is simply connected and $\theta\in H^{1,\infty}(\Omega)$ such that $|\theta(x)|=1$, then there exists a real function $\alpha\in H^{1,\infty}(\Omega)$ such that $e^{i\alpha}=\theta$, i.e.
$$
M_\theta=e^{iM\alpha},
$$
with $M_{i\alpha}\in\Gamma$.
\end{prop}
\begin{proof}
Consider the commutative Banach algebra $\a=C(\overline{\Omega},\mathbb{C})$ of  complex continuous maps in $\overline{\Omega}$, with the norm
$$
|f|_\infty =\sup_{x\in\overline{\Omega}}|f(x)|.
$$
Let $G_\a$ be the invertible group of $\a$. The maximal ideal spectrum of $\a$ is $\overline{\Omega}$, which by hypothesis is simply connected. Therefore, by the Arens-Royden Theorem \cite{royden},
\begin{equation}\label{exp a - r}
G_\a=\{e^g: g\in\a\}
\end{equation}
According to Lemma \ref{lips} the function $\theta$ is Lipschitz. Thus, $\theta$ can be extended to a continuous function in $\overline{\Omega}$.  By equation (\ref{exp a - r}) it follows  that $\theta=e^g$, for some continuous function $g$ on $\overline{\Omega}$. Since $|\theta(x)|=1$, it follows that $g=i\alpha$, with $\alpha$ real. In particular, note that $\alpha$ is bounded on $\Omega$.

 Moreover, we claim that $\alpha\in H^{1,\infty}(\Omega)$, and then it clearly follows  that $M_{i\alpha}$ belongs to $\Gamma$. To prove our claim, recall that $\theta$ is continuous, so that $|\theta(x) - \theta(y)| < 2$ if $\| x- y\|< \delta$, for some $\delta>0$. Therefore there is an analytic branch of the logarithm for all the points close enough to a fixed $x \in \Omega$. Then, $e^{i\alpha(y)}=e^{i \log(\theta(y))}$, and $\alpha(y)=\log(\theta(y)) + 2k \pi$ by  connectedness. Thus, $\alpha$ has partial derivatives almost everywhere, and
 \[  \big| \frac{\partial\alpha}{\partial x_j}(y) \big| = \big|  \frac{1}{\theta(y)} \, \frac{\partial\alpha}{\partial x_j}(y) \big|\leq \|\theta \|_{1,\infty} \, .  \]
 Since the¡is bound is the same for any point, we conclude that $\alpha \in H^{1,\infty}(\Omega)$.
\end{proof}

\medskip

\noi The same idea provides an example of an element in $\gg$ which is not in the range of the exponential, but it should be noted that in this example we consider $\Omega$ a compact manifold (rather than an open subset), namely $\Omega=\mathbb{T}$. See \cite[p. 232]{treves} for the definition of $H^1$ in this context.

\begin{ejem}\label{example exp}
Consider $\Omega=\mathbb{T}$, and the function $z$. We claim that $M_z\in\gg$ does not belong to the range of the exponential map. Suppose that $M_z=e^X$ for some $X\in\Gamma$. Then $X=M_g$ for some $g\in H^1$. Indeed, put $g=X1\in H^1$. Since $X$ commutes with $e^X=M_z$,
$$
Xz^n=X(M_z)^n1=(M_z)^nX1=z^ng=M_gz^n,
$$
for any integer $n$. It follows that $X=M_g$. Therefore $z=e^X1=e^g$, with $g$ continuous in $\mathbb{T}$, which is a contradiction.
\end{ejem}

\section{Stone's theorem in  $\gg$}\label{one parameter}
Clearly, the positive operator $A \in \b(H^1_0)$ such that    $[Af,g]=\PI{f}{g}$ is symmetrizable. According to Theorem \ref{symmetrizable results}, it extends to a bounded operator on $L^2$. Note that $A$   has dense range, both regarded as an operator on $L^2$ or $H^1_0$. The next elementary remark shows that $A(L^2)\subset H^1_0$. More precisely:

\begin{rem}\label{iso}
If $A$ is regarded as an operator in $\b(L^2)$, then $A^{1/2}(L^2)=H^1_0$. To this end, let $f\in L^2$ and $(g_n)_n$ be a sequence in $H^1_0$ such that $|g_n-f|_2\to 0$. Then $A^{1/2}g_n\to A^{1/2}f$ in $L^2$. Note that $A^{1/2}g_n$ is a Cauchy sequence in $H^1_0$. Indeed, we have
$$
|A^{1/2}(g_n-g_m)|^2_1=[g_n-g_m,A(g_n-g_m)]=\PI{g_n-g_m}{g_n-g_m}=|g_n-g_m|_2^2.
$$
It follows that $A^{1/2}f \in H^1_0$, and thus $A^{1/2}(L^2)\subset H^1_0$. On the other hand, if $g\in H^1_0$, since $A^{1/2}$ has dense range, there exists a sequence $(f_n)_n$ in $H_0^1$ such that $|A^{1/2}f_n-g|_1\to 0$. The same computation above shows that $(f_n)_n$ is a Cauchy sequence in $L^2$:
$$
|f_n-f_m|_2 ^2=|A^{1/2}(f_n-f_m)|^2_1.
$$
Therefore there exists $f\in L^2$ such that $|f_n - f|_2\to 0$. Then $A^{1/2}f=g$. Moreover, we have
$$
|A^{1/2}f|^2_1=[f,Af]=\langle f,f\rangle =|f|_2^2 \, .
$$
Hence $A^{1/2}:L^2\to H^1_0$ is a surjetive isometry.
\end{rem}

\begin{nota}
The surjective isometry   $A^{1/2}:(L^2, | \, \cdot \, |_2)\to (H^1_0, | \, \cdot \,  |)$ will be denoted by $\mathcal{A}^{1/2}$  to distinguish it from the operator $A^{1/2}$ acting on  $L^2$ or  $H^1_0$.
\end{nota}

There is yet another characterization of $\gg$:
\begin{prop}\label{usubge}
Let $G$ be an invertible operator on $H^1_0$. There is a unique unitary operator $U_G\in U(H^1_0)$ such that $U_GA^{1/2}=A^{1/2}G$ if and only if $G\in \gg$. The map $G\mapsto U_G$ is a group isomorphism from $\gg$ onto the group
$$
{\cal U}_{R(A^{1/2})}(H^1_0)=\{U\in U(H^1_0): U(R(A^{1/2}))=R(A^{1/2})\}.
$$
Moreover, if $Ad_{\mathcal{A}^{1/2}}:\b(L^2)\to\b(H^1_0)$ denotes the $C^*$-algebra isomorphism implemented by the unitary transformation $\mathcal{A}^{1/2}$, $Ad_{\mathcal{A}^{1/2}}(X)=\mathcal{A}^{1/2} X\mathcal{A}^{-1/2}$, then
$$
Ad_{\mathcal{A}^{1/2}}(\{U\in U(L^2): U(H^1_0)=H^1_0\})={\cal U}_{R(A^{1/2})}(H^1_0).
$$
\end{prop}
\begin{proof}
The only if part is algebraic: if there exists such a $U_G$, then
$$
G^*AG=A^{1/2}U_G^*U_GA^{1/2}=A,
$$
thus $G\in \gg$. To prove the other implication, note that when $G\in \gg$, the operator $A^{1/2}G$ is injective and has dense range. Therefore the isometric part in its polar decomposition $U|A^{1/2}G|$ extends to a unitary operator, which we denote  $U=U_G$. Note that
$$
|A^{1/2}G|=((A^{1/2}G)^*A^{1/2}G)^{1/2}=(G^*AG)^{1/2}=A^{1/2},
$$
and thus $U_GA^{1/2}=A^{1/2}G$. The unitary $U_G$ is clearly unique with this property, and the mapping $G\mapsto U_G$ is a group homomorphism:
$$
A^{1/2}G_1G_2=U_{G_1}U_{G_2}A^{1/2}=U_{G_1G_2}A^{1/2}.
$$
Clearly $U_GA^{1/2}=A^{1/2}G$ implies that $U_G(R(A^{1/2}))\subset R(A^{1/2})$. Since the same is true for $U_{G^{-1}}=U_G^{-1}$, equality holds. Pick a unitary operator $U$ on $H^1_0$ such that $U(R(A^{1/2}))=R(A^{1/2})$. For $f\in H^1_0$, $UA^{1/2}f\in R(A^{1/2})$, since $A^{1/2}$ is injective, there exists $g_f\in H^1_0$ such that
$UA^{1/2}f=A^{1/2}g_f$. Thus a map $f\mapsto g_f$ is defined, which is clearly a linear bijection of $H^1_0$. Moreover,
since $U$ is unitary in $H^1_0$,
$$
|f|_2=|A^{1/2}f|_1=|UA^{1/2}f|_1=|A^{1/2}g_f|_1=|g_f|_2.
$$
Thus $f\mapsto g_f$ extends to a unitary operator in $L^2$, which by construction fixes $H^1_0$, thus $G$ defined $Gf=g_f$, belongs to $\gg$, and clearly $UA^{1/2}=GA^{1/2}$.

Finally, a straightforward verification shows that $\mathcal{A}^{-1/2}U_G\mathcal{A}^{1/2}$ is a unitary operator of $L^2$, which extends $G$.
\end{proof}

\begin{coro}
The map $\gg\ni G\mapsto U_G\in {\cal U}_{R(A^{1/2})}$ from Proposition \ref{usubge} is a continuous group isomorphism (in the topology induced by the norm of $\b(H^1_0)$). Its inverse is not continuous.
\end{coro}
\begin{proof} Modulo the automorphism $Ad_{\cal A}$, this homomorphism is the extension map $G\mapsto \bar{G}$ (see Proposition \ref{desi}).
\end{proof}

Let $G(t)$ be a  strongly continuous one parameter group in $\gg$, i.e. for $t\in\mathbb{R}$, $G(t)\in\gg$, $G(0)=1$, $G(t+s)=G(t)G(s)$, and for each $f\in H^1_0$, the map $t\mapsto G(t)f\in H^1_0$ is continuous.
By Proposition \ref{usubge}, this gives rise to a one parameter group of unitaries $U_{G(t)}$. Let us see first that $U_{G(t)}$ is also strongly continuous.
\begin{prop}
Let $G(t)$, $t\in\mathbb{R}$, be a strongly continuous one parameter group in $\gg$. Then $U_{G(t)}$ is a strongly continuous group of unitaries in $H^1_0$.
\end{prop}
\begin{proof}
Since the map $G\mapsto U_G$ is a group homomorphism, it is clear that $U_{G(t)}$ is a one parameter group of unitaries. The fact that $U_{G(t)}A^{1/2}=A^{1/2}G(t)$, implies that $t\mapsto U_{G(t)}f$ is continuous for any $f\in R(A^{1/2})$, which is dense in $H^1_0$. By von Neumann's extension of Stone's theorem  (see for instance  \cite[Theorem VIII.9]{reedsimon}), which states that a one parameter group in a separable Hilbert space, which is weakly measurable, is strongly continuous, our result follows. Indeed, if $f\in H^1_0$, let $(f_n)$ be a sequence in $R(A^{1/2})$ such that $|f_n-f|_1\to 0$. Then, for each $t\in\mathbb{R}$ and $g\in H^1_0$,
$$
\phi_n(t)=[U_{G(t)}f_n,g]\to \phi(t)=[U_{G(t)}f,g].
$$
Since the $\phi_n$ are continuous, it follows that $\phi$ is measurable.
\end{proof}
According to Stone's theorem, there exists a (possibly unbounded) self-adjoint operator $S:D(S)\subset H^1_0\to H^1_0$ such that
$$
U_{G(t)}=e^{itS}.
$$
Let us now relate $S$ to the Lie algebra $\Gamma$.
\begin{rem}
If $G(t)$ is strongly continuous differentiable, i.e. if $t\mapsto G(t)f$ is continuously differentiable for every $f \in H^1_0$, then the identity $U_{G(t)}A^{1/2}=A^{1/2}G(t)$ implies that the function $t\mapsto U_{G(t)}f$ is  differentiable for any $f\in R(A^{1/2})$. Thus $R(A^{1/2})\subset D(S)$, and moreover, $iSA^{1/2}f=A^{1/2}\dot{G}(0)f$.

On the other hand, $\dot{G}(0)$ is an everywhere defined operator in $H^1_0$, which  has an adjoint. Indeed, $G^*(t)$ is weakly differentiable, because $t\mapsto [G^*(t)f,g]=[f,G(t)g]$ is differentiable, and therefore its weak derivative $\dot{G}^*(0)$ is an adjoint for $\dot{G}(0)$. It follows that $\dot{G}(0)$ is bounded. Thus, differentiating the identity
$$
G^*(t)AG(t)f=Af
$$
at $t=0$ for any $f\in H^1_0$, yields
$$
\dot{G}(0)^*Af+A\dot{G}(0)f=0.
$$
Therefore $\dot{G}(0)\in \Gamma$. Moreover,  $X=i\dot{G}(0)$ satisfies $X^*A=AX$. By  Theorem \ref{Quasiherm}  there exists a self-adjoint bounded operator $S_0$ in $H^1_0$ such that $S_0A^{1/2}=A^{1/2}X$. Therefore, $S_0=S$, that is, $S$ is bounded, and satisfies $S(R(A^{1/2}))\subset R(A^{1/2})$.
\end{rem}
In the general case ($G(t)$ strongly continuous), we have the following result. Let $C_0^\infty(\mathbb{R})$ denote the space of smooth functions with compact support on $\mathbb{R}$. Let  $D \subset H_0^1$ be the  linear span of the vectors
$$
f_\varphi=\int_{-\infty}^\infty \varphi(t) G(t)f d t,
$$
where $f\in H^1_0$ and $\varphi\in C_0^\infty(\mathbb{R})$.    This space $D$ was used in the proof of Stone's theorem   due to G\r{a}rding and Wightmann (see \cite{reedsimon}). The fact that $D$  is dense in $H_0^1$ is a general property of the space $D$ for any underlying Hilbert space.
\begin{prop}
With the above notations, the following assertions hold:
\begin{enumerate}
\item[i)]
The subspace $D$ is dense in $H^1_0$, satisfies that $G(t)(D)\subset D$ for all $t\in\mathbb{R}$, and for $f\in D$, $t\mapsto G(t)f$ is differentiable.
\item[ii)]
The subspace $A^{1/2}(D)$ is dense in $H^1_0$,  satisfies that $U_G(t)(A^{1/2}(D))\subset A^{1/2}(D)$, and for $f\in A^{1/2}(D)$, the map $t\mapsto U_{G(t)}f$ is differentiable.
\item[iii)]
$S(A^{1/2}(D))\subset A^{1/2}(D)$, $A^{1/2}(D)$ is a core for $S$, and if $f\in D$,
$$
iSA^{1/2}f=A^{1/2}\dot{G}(0)f.
$$
\end{enumerate}
\end{prop}
\begin{proof}
Pick $\varphi\in C_0^\infty(\mathbb{R})$ and $f\in H^1_0$, then,
$$
G(t)f_\varphi=\int_{-\infty}^\infty \varphi(s)G(s+t)f d s=\int_{-\infty}^\infty \varphi(s-t)G(s)f d s=f_{\varphi(\,\cdot \,-t)}\in D,
$$ and clearly $t \mapsto G(t)f_\varphi$ is differentiable. Since $D\subset H^1_0$ is dense, and $A$ has dense range, then $A^{1/2}(D)$ is dense in $H^1_0$. Moreover,
$$
U_{G(t)}A^{1/2}f_\varphi=A^{1/2}G(t)f_\varphi=A^{1/2}f_{\varphi(\ -t)}\in A^{1/2}(D),
$$
and it is also clearly differentiable as a function in $t$. Therefore
$$
e^{itS}(A^{1/2}(D))=U_{G(t)}(A^{1/2}(D))\subset A^{1/2}(D)
$$
 for all $t\in \mathbb{R}$. This implies that $A^{1/2}(D)$ is a core for $S$ (see \cite[Theorem VIII.11]{reedsimon}). Finally, differentiating $U_{G(t)}A^{1/2}f_\varphi=A^{1/2}G(t)f_\varphi$ at $t=0$, one obtains
 $iSA^{1/2}f_\varphi=A^{1/2}\dot{G}(0)f_\varphi$.
\end{proof}

\section{Invariant Finsler metrics in $\gg$}\label{Finsler metrics}

The group $\gg$ preserves the norms in $L^2$, the usual spectral norm and the Schatten $p$-norms. Therefore it is natural, from a geometric standpoint, to consider these norms to endow $\gg$ with a Finsler metric.
The tangent space $(T\gg)_G$ of $\gg$ at $G$ identifies with
$$
(T\gg)_G=g\Gamma=\{GX: X\in \Gamma\}.
$$
Since the elements $G\in\gg$ preserve the $2$-norm $\|\ \|_2$, it is natural to consider, in each tangent space, the norm
$$
\|V\|_G=\|V\|_{\b(L^2)}.
$$
Note that if $V=GX$, for $X\in\Gamma$, then
$$
\|V\|_G=\|GX\|_{\b(L^2)}=\|\bar{G}X\|_{\b(L^2)}=\|X\|_{\b(L^2)},
$$
because $\bar{G}$ is a unitary operator in $L^2$. This implies that this metric is bi-invariant for the left and right action of $\gg$ on itself. Note that the tangent spaces are not complete with this norm.

We measure the length of a differentiable  curve $\gamma$ in $\gg$, parametrized in the interval $I$, as is usual, by
$$
L_2(\gamma)=\int_I \|\dot{\gamma}(t)\|_{\gamma(t)} dt=\int_I \|\dot{\gamma}(t)\|_{\b(L^2)} dt.
$$
The rectifiable distance $d_2$ induced by the infima of the $L_2$-length of the paths joining given endpoints is a continuous map when we give $\gg$ the natural topology as an open subset of the bounded linear operators on $H^1_0$. However, the topology induced by this rectifiable distance on $\gg$ is finer thus what we have introduced is a weak Finsler metric on the manifold $\gg$.

\begin{prop}\label{minimal curves in G}
Suppose that $G\in\gg$ with $\|G-1\|_{\b(H^1)}\le 1$. Then there exists a curve $\delta(t)=e^{tX}$, with $X\in\Gamma$ such that $\delta(1)=G$, which has minimal length among all curves in $\gg$ joining $1$ and $G$, and in particular $d_2(1,G)=\|X\|_{{\mathcal B}(L^2)}$.
\end{prop}
\begin{proof}
By Lemma \ref{lemaexponencial}, there exists $X\in\Gamma$ such that $e^X=G$. Moreover, $X=log(G)$, with $log$ being the branch of the logarithm with singularities in the negative real axis.
By the formula (\ref{espectros}),
$$
\sigma_{L^2}(\bar{G})\subset \sigma_{H^1_0}(G)\cup \overline{\sigma_{H^1_0}(G^{-1})}.
$$
Note that $\|G-1\|\le 1$ implies that $\sigma_{H^1_0}(G)\subset \{z\in\mathbb{C}: Re(z)\ge 0\}$. Then, if $\lambda \in \sigma_{H^1_0}(G^{-1})$, $\lambda=\mu^{-1}$ with $\mu\in \sigma_{H^1_0}(G)$, and thus $Re(\lambda)\ge 0$. It follows that
$$
\sigma_{L^2}(\bar{G})\subset \{z\in\mathbb{C}: Re(z)\ge 0\}\cap \mathbb{T}=\{e^{i\theta}: |\theta|\le \pi/2\},
$$
and therefore $\|X\|_{\b(L^2)}\le \pi/2$. Note that $L_2(\gamma)$ equals the length of the curve of unitaries  in $L^2$, measured with the Finsler metric given by the usual operator norm on $\b(L^2)$. It is a known fact that one parameter groups of unitaries $e^{tX}$ have minimal length along their paths, for time $t$ such that $|t|\|X\|_{\b(L^2)}\le \pi$ (see for instance \cite{al}). Therefore $\delta(t)$ remains minimal  for $|t|\le 2$, which proves our assertion.
\end{proof}

\subsection{The subgroups $\gg_p$}
Let $\b_p(H^1)$,  $1\leq p \leq \infty$, be the Schatten ideals of operators on $H^1_0$. As usual, $\b_\infty(H^1_0)$ stands for the compact operators on $H^1_0$. We introduce the following subgroups:
\[  \gg_p :=\gg \cap (I- \b_p (H^1_0)).  \]
 Clearly,  $\gg_p\subset\gg_\infty$ properly. Apparently, the Banach-Lie algebra $\Gamma_p$ of $\gg_p$ is
$$
\Gamma_p=\Gamma\cap \b_p(H^1_0).
$$



For the subgroup $\gg_\infty$ there is a stronger result on the minimality of curves.  First note that using Lemma \ref{lemaexponencial}, we obtain that any $G\in\gg_\infty$ is of the form $G=e^X$ for some compact operator $X\in\Gamma$.
\begin{prop}\label{exponente compacto}
An operator $G$ belongs to $\gg_\infty$ if and only if there exists a compact operator $X\in\Gamma$ such that $e^X=G$
\end{prop}
\begin{proof}
The sufficient part is clear. Note that the spectrum $\sigma_{H^1_0}(G)$ is finite, or a sequence in $\mathbb{T}$ converging to $1$. In particular one can always find a half-line ${\cal L}$ connecting $0$ and infinity, which does not intersect $\sigma_{H^1_0}(G)$. Thus by  Lemma \ref{lemaexponencial}, there exists  $X\in \Gamma$ such that $e^X=G$,
namely
$$
X=\frac{1}{2\pi i}\int_\alpha log(z)(G-z1)^{-1} d z,
$$
where $\alpha$ is a simple closed curve which does not intersect ${\cal L}$ and encompasses $\sigma_{H^1_0}(G)$. Note that  since $1\in \sigma_{H^1_0}(G)$, one can adjust the definition of log in a way such that $0\in \sigma_{H^1_0}(X)$ (trimming eigenvalues which are multiples of $2\pi i$).
It remains to prove that such $X$ is compact. Note that $G-z1$ belongs to the Banach algebra $\mathbb{C}1+\b_\infty(H^1_0)$, the unitization of the algebra $\b_\infty(H^1_0)$ of compact operators. Therefore (since $log(z)(G-z1)^{-1}$ is a continuous map in $z$, defined on a neighborhood of $\sigma_{H^1_0}(G)$), it follows that $X\in \mathbb{C}1+\b_\infty(H^1_0)$, i.e. $X=\lambda +K$. Since $X$ is non-invertible, it must be $\lambda=0$.
\end{proof}

In particular, if $G_1,G_2\in\gg_\infty$, then there exists a compact operator $X\in\Gamma$ such that $G_2=e^XG_1$.

\begin{teo}\label{thm minimal curves}
Let $G_1,G_2\in\gg_\infty$.  Then there exists a compact operator $X\in \Gamma$ such that the curve $\delta(t)=e^{tX}G_1$ verifies $\delta(1)=G_2$ and has minimal length among all curves joining the same endpoints in $\gg_\infty$ (and in $\gg$).
\end{teo}
\begin{proof}
By the above Proposition, there exists  $X\in\Gamma$, which is compact and verifies that $e^XG_1=G_2$. If $\|X\|_{\b(L^2)}\le \pi$, the result follows using the same argument as in Proposition \ref{minimal curves in G}, with the (unitary extension of the) curve $\delta(t)=e^{tX}G_1$. Suppose otherwise, then there exist finitely many eigenvalues $\lambda$ of $X$, such that $|\lambda|>\pi$. Pick one such $\lambda$. If $P$ is the spectral projection in $L^2$ associated to $\lambda$, then $P(H^1_0)\subset H_0^1$. Indeed, the eigenvectors of the extension of $X$ to $L^2$ belong to $H^1_0$ (see Theorem \ref{symmetrizable results} $iii)$). Therefore, $iP|_{H^1_0}\in\Gamma$. There exists an integer $m$ such that $|\lambda-2m\pi i|\le\pi$. Then $X'=X-2m\pi i P\in\Gamma$ is compact, and clearly verifies $e^{X'}=e^X$. Replacing in this fashion all the eigenvalues (finite in number) which lie outside $[-\pi,\pi]$ yields a compact operator $X_0$ such that $e^{X_0}G_1=G_2$
\end{proof}

\begin{prop}\label{p normas}
If $G\in\gg_p$, then there exists $X\in\Gamma_p$ such that $\|X\|\le \pi$ and  $e^X=G$.
\end{prop}
\begin{proof}
By Proposition \ref{exponente compacto}, there exists a compact operator $X\in\Gamma$ with $\|X\|\le \pi$ such that $e^X=G$. It remains to prove that $X\in\b_p(H^1_0)$. The proof is similar to that case: consider now the Banach algebra $\mathbb{C}1+\b_p(H^1_0)$,  the unitization of $\b_p(H^1_0)$, which is a Banach algebra with the $p$-norm. Since $log(z)(G-z1)^{-1}$ is continuous in the $p$-norm topology, it follows  $X=\lambda 1+K$, with $K\in\b_p(H^1_0)$. Again, since $0\in\sigma_{H^1_0}(X)$, $\lambda=0$.
\end{proof}
Since $\Gamma_p\subset \b_p(H^1_0)$, a natural metric to consider in $\gg_p$, which takes account of the specific spectral properties of the elements in $\gg_p$, should be related to the $p$-norm. On the other hand, as remarked at the beginning of this section, we want the metric to be invariant by the action of the group. By Theorem \ref{symmetrizable results}, the operators $X\in\Gamma_p$, when extended to $L^2$, are compact and antihermitic. Moreover, the eigenvalues and multiplicities of the extension remain the same as for $X$. By a classical inequality of Lalesco \cite{lalesco} (see also \cite{simon}), the $p$-norm of the sequence of eigenvalues of $X$ is bounded by the $p$-norm of the sequence of singular values of $X$. The former equals the $p$-norm of the extension of $X$ to $L^2$ (because $X$ is antihermitic there),  the latter is the $p$-norm of $X$ in $H^1_0$. Thus,
\begin{equation}\label{des entre los dos ideales}
\|X\|_{p,\b(L^2)}\le \|X\|_p.
\end{equation}
We define the following metric in $\gg_p$: if $X\in (T\gg_p)_G$, then
$$
\|X\|_{p,G}=\|X\|_{p,\b(L^2)}.
$$
\begin{teo}
Let $G_1, G_2\in \gg_p$. Then there exists $X\in\Gamma_p$ such that the curve $\delta(t)=e^{tX}G_1$ in $\gg_p$, verifies $\delta(1)=G_2$ and has minimal length for the above defined metric, among all smooth curves joining the same endpoints in $\gg_p$.
\end{teo}
\begin{proof}
By Proposition \ref{p normas}, there exists $X\in\Gamma_p$ with $\|X\|_{\b(L^2)}\le \pi$ such that $G_2=e^XG_1$. The result now follows as with $\gg_\infty$, using the  that in the classical unitary groups
$$
U_p(L^2)=\{G\in U(L^2): G-1\in\b_p(L^2)\},
$$
curves of the form $e^{tX}$, where $X$ is antihermitic,  have minimal length for the $p$-norm for $|t|\le 1$ provided that $\|X\|_{\b(L^2)}\le \pi$ (see \cite{alr} for the case $p\ge 2$ or \cite{alv} for the general case).
\end{proof}

\begin{rem}
Using Lalesco's inequality, one may prove  the inequality in (\ref{des entre los dos ideales}) for any symmetric norm  in the sense of \cite{simon}. Moreover, our last result on minimality of curves can be carried out in the general setting of symmetrically normed ideals.
\end{rem}

\medskip
\medskip

\noi
\begin{tabular}[h]{ll}
			 Esteban Andruchow and Gabriel Larotonda &  Eduardo Chiumiento \\
			 Instituto de Ciencias, UNGS & Dto. de Matem\'atica, FCE-UNLP \\
			 J. M. Gutierrez 1150 & Calles 50 y 115\\
						  (B1613GSX) Los Polvorines, Argentina   &  (1900) La Plata, Argentina \\
			   \textit{e-mails}: \texttt{eandruch@ungs.edu.ar,} & \textit{e-mail}: \texttt{eduardo@mate.unlp.edu.ar}\\
			   \texttt{glaroton@ungs.edu.ar}  & \\
			  & \\
			  & \\
			  E. Andruchow, E. Chiumiento and G. Larotonda &\\
        Instituto Argentino de Matem\'atica& \\
        ``Alberto P. Calder\'on'', CONICET & \\
        Saavedra 15  Piso 3& \\
                (1083) Buenos Aires, Argentina & \\
\end{tabular}

\end{document}